\documentclass[a4paper,10pt,oneside]{amsart}
\usepackage{amssymb, amsmath, amsthm, mathrsfs, amsfonts, latexsym, stmaryrd, thmtools}

\usepackage[shortlabels]{enumitem}

\usepackage{url}
\usepackage{graphicx}
\usepackage{xcolor}

\definecolor{svlinks}{rgb}{.0,0.3,0.6}

\usepackage{fouriernc}
\usepackage[margin=3cm]{geometry}
\usepackage[utf8]{inputenc}

\swapnumbers
\newtheorem{observation}[subsection]{Observation}
\newtheorem*{observation*}{Observation}
\newtheorem{proposition}[subsection]{Proposition}
\newtheorem*{proposition*}{Proposition}
\newtheorem{theorem}[subsection]{Theorem}
\newtheorem*{theorem*}{Theorem}
\theoremstyle{definition}
\newtheorem{definition}[subsection]{Definition}
\newtheorem*{definition*}{Definition}

\newtheorem*{claim*}{Claim}

\newtheorem*{question*}{Question}
\newtheorem{note}[subsection]{Note}
\newtheorem*{note*}{Note}
\newtheorem{fact}[subsection]{Fact}
\newtheorem*{fact*}{Fact}

\usepackage[hidelinks]{hyperref}
\begin{document}
\makeatletter\def\shfiuwefootnote{\gdef\@thefnmark{}\@footnotetext}\makeatother\shfiuwefootnote{Version 2022-10-23\_3. See \url{https://shelah.logic.at/papers/1123/} for possible updates.}

\title{Ramsey Partitions of Metric Spaces}

\author{Saharon Shelah}
\address{The Hebrew University of Jerusalem and Rutgers University.}
\thanks{The first author was partially supported by European Research Council grant 338821. The paper has been edited using typing services generously funded by an individual who wishes to remain anonymous. Paper 1123 on author's list.}
\author{Jonathan Verner}
\address{Department of Logic, Faculty of Arts, Charles University, Prague}
\thanks{The second author was supported by FWF-GA\v{C}R LA Grant no.\ 17-33849L: Filters, ultrafilters and connections with
forcing. }

\begin{abstract}
We investigate the existence of metric spaces which, for any coloring with a fixed number of colors, contain monochromatic isomorphic copies of a fixed starting space K. In the main theorem we construct such a space of size \(2^{\aleph_0}\) for colorings with \(\aleph_0\) colors and any metric space \(K\) of size \(\aleph_0\). We also give a slightly weaker theorem for countable ultrametric \(K\) where, however, the resulting space has size~\(\aleph_1\).
\end{abstract}

\maketitle

\section{Introduction}

Recall that the standard Hungarian arrow notation
 \[
\kappa\rightarrow(\lambda)^\nu_\mu
\]
says that whenever we color \(\nu\)-sized subsets of \(\kappa\) with \(\mu\)-many
colors there is a homogeneous subset of \(\kappa\) of size \(\lambda\). 
The question whether, for a given \(\lambda,\nu,\mu\), there is a \(\kappa\)
such that the arrow holds has been well studied in Ramsey theory. If \(\nu=1\) 
the coloring becomes a partition of \(\kappa\) and the question reduces to a 
simple cardinality argument. However, if we add additional structure into the mix,
the question becomes nontrivial. The following definition makes precise what we 
mean by ``adding additional structure'':

\begin{definition*}
Let  \(\mathcal K\) be a class of structures and \(\kappa,\lambda,\mu\)
be cardinals. The arrow
\[
\kappa\rightarrow_{\mathcal K}(\lambda)^1_\mu,
\]
is shorthand for the statement that for every \(K\in\mathcal K\) of size 
\(\lambda\) there is a \(Y\in\mathcal K\) of size \(\kappa\) such that 
for any partition of \(Y\) into \(\mu\)-many pieces one of the pieces contains
an isomorphic copy of \(K\).
\end{definition*}

Note that for a class of structures there are often several natural notions
of  \emph{contains an isomorphic copy}. So the above notation assumes that the choice
of \(\mathcal K\) includes choosing the notion of \emph{contains an isomorphic copy}. 
The basic question, given a class \(\mathcal K\), then becomes whether for every 
\(\lambda,\mu\) there is a \(\kappa\) such that 
\(\kappa\rightarrow_{\mathcal K}(\lambda)^1_\mu\).

These types of questions have been considered before. For example A. Hajnal
and P. Komj\'{a}th consider (\cite{Hajnal:1988}, see also \cite{Sh:289}) 
the class \(\mathcal G\) of 
well-ordered undirected graphs. The notion of ``G contains an isomorphic copy of H'' 
is ``G contains an induced subgraph graph-isomorphic to H via an order-preserving
bijection''. For this class they prove

\begin{theorem*}[Hajnal, Komj\'{a}th]
\[
2^\kappa\rightarrow_{\mathcal G}(\kappa)^1_\kappa.
\]
\end{theorem*}

J. Ne\v{s}et\v{r}il and V. R\"{o}dl consider (\cite{Nesetril:1977}) the classes \(\mathcal T_0\) and \(\mathcal T_1\)
of all \(T_0\) and \(T_1\) topological spaces with homeomorphic embeddings.
They prove

\begin{theorem*}[Ne\v{s}et\v{r}il and V. R\"{o}dl]
If  \(\mathcal T = \mathcal T_0\) or \(\mathcal T = \mathcal T_1\) then
\[
\kappa^\gamma\rightarrow_{\mathcal T}(\kappa)^1_\gamma
\]
\end{theorem*}

In this paper we will be interested mainly in these questions for metric spaces.
There have been some results for metric spaces (see e.g.  \cite{Komjath:1987},
\cite{Nesetril:2006}, \cite{Komjath:1987b}, \cite{Weiss:1990}). 
Most notably, W. Weiss shows in \cite{Weiss:1990} that there is a limit to what 
one can prove:

\begin{theorem*}[Weiss]
Assume that there are no inner models with measurable cardinals. If  \(X\) is a 
topological space then there is a coloring of \(X\) by two colours such that
\(X\) doesn't contain a monochromatic homeomorphic copy of the Cantor set.
\end{theorem*}

Also see \cite[3.8(1), 3.9(3)]{Sh:460}: it says that if $2^{\aleph_0} > \aleph_\omega$ and some very weak statement holds (the precise formulation is unimportant here, but it is weak enough that the consistency of its negation holds) then $X$ can be divided into $2^{\aleph_0}$ many equally large sets such that the intersection of any two of them has 
small cardinality. 
In particular, this holds in \(\mathbf V^{\mathbb P}\) if \(\mathbb P\) adds 
$\geq \beth_\omega$ Cohen reals. See more in \cite{Sh:668}.

In particular in the class of metric spaces, we can't hope for positive results
if \(\kappa >\omega\) (but see \cite{Sh:668} for a positive result from
a supercompact cardinal; more history can be found there). The case \(\kappa=\omega\)
is not ruled out and, in fact, the main result of this paper, due to the first
author, is a positive arrow for this case.

\begin{definition} 
Let \(\mathcal M\) be the class of 
bounded metric spaces with ``\(X\) contains an isomorphic copy of \(Y\)'' being 
``\(X\) contains a subspace which is a scaled copy of \(Y\)''. (\(K\) is a scaled
copy of \(Y\) if there is a bijection \(f:K\to Y\) onto \(Y\) and a scaling factor
\(c\in\mathbb R^+\) such that \(d_K(x,y)=c\cdot d_Y(f(x),f(y))\). 
\end{definition}

\begin{theorem}
\label{thm.main} 
\[
2^\omega\rightarrow_{\mathcal M}(\omega)^1_\omega.
\]
\end{theorem}

In fact the theorem we prove is much stronger: for every countable metric space
any \(\aleph_1\)-saturated metric space \(X\) works.

The original motivation of the second author for considering these arrows comes
from a problem of M. Hru\v{s}\'{a}k stated in (\cite{Hrusak:2012}):

\begin{question*}
Does ZFC prove that there is a non  \(\sigma\)-monotone metric space
of size \(\aleph_1\)?
\end{question*}

If one could replace  \(2^\omega\) by \(\aleph_1\) in the above arrow, this
would give a positive answer. In fact, for a positive answer it would be sufficient
to consider the class \(\mathcal M\) with isomorphic copies being Lipschitz images,
which seems to be much weaker.

The paper is organized as follows. In the second section we prove the main result
and in the third section we discuss what can be proved for the restricted class
of ultrametric spaces. We finish the introduction by recalling some definitions 
and facts for the benefit of the reader.

\begin{definition}\label{def:metric}
1) A \emph{metric space} is a pair \((X,\rho)\) where 
\(\rho:X\times X\to\mathbb R\) is a \emph{metric} (on \(X\)), i.e. it satisfies, 
for all \(x,y,z\in X\),

\begin{enumerate}[(a)]
  \item  \(\rho(x,y)\geq 0\) and \(\rho(x,y)=0\iff x=y\);
  \item  \(\rho(x,y)=\rho(y,x)\); and
  \item  \(\rho(x,z)\leq\rho(x,y)+\rho(y,z)\).
\end{enumerate}

2) The third condition is called the  \emph{triangle inequality}. 
If it is strengthened to
\[
  \label{eq:strongtriangle} \forall x,y,z \in X,\  \rho(x,z) \leq \max\{\rho(x,y),\rho(y,z)\}
\]
\underline{then} we say that the space is \emph{ultrametric}.


3) In the remainder of this paper, we may abuse notation slightly 
and refer to the metric space $(X,\rho)$ as $X$.
\end{definition}

%
%
%

\begin{definition}\label{def:saturated}
 A metric space \((X,\rho)\) is \(\aleph_1\)-saturated if
 for any at most countable \(Y\subseteq X\) and any function \(f:Y\to\mathbb R^+\) satisfying
 the triangle inequality
 \[
  \tag{$*$}\label{eq:triangle}
  f(x)+f(y)\geq \rho(x,y)\ \&\ f(x)+\rho(x,y)\geq f(y)
 \]
 for all \(x,y\in Y\) there is \(p\in X\) such that \(\rho(x,p) = f(x)\) for
 all \(x\in Y\).
 \end{definition}
 
 \begin{note}
 There is a standard way to see $X$ as a structure for a language with countably many 
 binary predicates \(\{R_q : q \in \mathbb Q\}\): namely, interpret the predicate 
 \(R_q(x,y)\) as \(\rho(x,y) \leq q\). Then the space $X$ is $\aleph_1$-saturated if 
 \begin{enumerate}
     \item it contains a copy of every finite metric space,
     
     \item given any finite metric spaces $Y_1 \subseteq Y_2$ with $|Y_2 \setminus Y_1| = 1$ and an isometry $\pi : Y_1 \to X$, the isometry can be extended to $Y_2$, and
     
     \item every bounded 1-type over a countable subset of $X$ is realized.
 \end{enumerate}
 \end{note}
 
 The following is standard and is included here for the convenience of the reader.
 
 \begin{observation} There is an \(\aleph_1\)-saturated metric space of size
 \(2^\omega\).
 \end{observation}
 \begin{proof} Let \(\{(Y_\alpha,f_\alpha):\alpha<2^\omega\}\) be an enumeration
 of all pairs such that \(Y_\alpha\in[2^\omega]^{\leq\omega}\) and \(f_\alpha:Y_\alpha\to\mathbb R^+\)
 with each pair appearing cofinally often. By induction define a sequence
 \(\langle d_\alpha:\alpha<2^\omega\rangle\) such that
 \begin{enumerate}
    \item \label{eq:extend} \(d_\alpha\subseteq d_\beta\) for all \(\alpha<\beta<2^\omega\);
    
    \item \label{eq:metric} \(d_\alpha\) is a metric on \(\alpha\); and
    
    \item \label{eq:realize} if \(Y_\alpha\subseteq\alpha\) and \((Y_\alpha,f_\alpha)\) satisfies \eqref{eq:triangle} of \ref{def:saturated} and there is no 
    \(\beta<\alpha\) such that \(d_{\alpha}(y,\beta)=f_\alpha (y)\) for
    all \(y\in Y_\alpha\) then \(d_{\alpha+1}(y,\alpha)=f_\alpha(y)\) for all 
    \(y\in Y_\alpha\).
 \end{enumerate}
 The only nontrivial part is guaranteeing \eqref{eq:realize} for successors. So
 assume \(Y_\alpha\subseteq\alpha\) and that \eqref{eq:triangle} is satisfied
 and for each \(\beta<\alpha\) there is \(y\in Y_\alpha\) such that \(d_{\alpha}(\beta,y)\neq f_\alpha(y)\).
 Extend \(d_\alpha\) to \(d_{\alpha+1}\) by defining
 \[
   d_{\alpha+1}(\beta,\alpha) = \inf\ \big\{d_\alpha(\beta,y) + f_\alpha(y) : 
   y\in Y_\alpha\big\},\quad d_{\alpha+1}(\alpha, \alpha) = 0.
 \]
 Then clearly both \eqref{eq:extend} and \eqref{eq:realize} are satisfied. To
 show that \eqref{eq:metric} is satisfied it is enough to show that \(d_{\alpha+1}(\beta,\alpha)>0\)
 for all \(\beta<\alpha\). Assume this is not the case for some \(\beta<\alpha\). 
 By assumption there is  \(y\in Y_\alpha\) such that 
 \(0<|f_\alpha(y)-d_\alpha(\beta,y)| = \varepsilon\). Since 
 \(d_{\alpha+1}(\beta,\alpha)=0\) we can find \(z\in Y_\alpha\) such that 
 \(d_\alpha(\beta,z)+f_\alpha(z) < \varepsilon/2\). There are two cases, 
 both leading to a contradiction: if \(f_\alpha(y) > d_\alpha(\beta,y)\) then \(d_\alpha(z,y)<d_\alpha(\beta,y)+\varepsilon/2\) 
 so \(d_\alpha(z,y) + f_\alpha(z) < d_\alpha(\beta,y) + \varepsilon = f_\alpha(y)\) 
 contradicting \eqref{eq:triangle}.
 On the other hand if \(f_\alpha(y) < d_\alpha(\beta,y)\) then 
 \(d_\alpha(z,y) \geq d_\alpha(\beta,y) - d_\alpha(\beta,z) =
 f_\alpha(y) + \varepsilon - d_\alpha(\beta,z) > f_\alpha(y) + \varepsilon/2 \geq f_\alpha(y) + f_\alpha(z)\) again contradicting \eqref{eq:triangle}.
 This completes the inductive definition. Finally we show that \((2^\omega,d_{2^\omega})\) is
 \(\aleph_1\)-saturated. Fix an at most countable \(Y\subseteq 2^\omega\) and an
 \(f:Y\to\mathbb R^+\). Find \(\alpha< 2^\omega\) such that \(Y\subseteq\alpha\) and
 \((Y,F)=(Y_\alpha,f_\alpha)\). But then the existence of \(p\) in \autoref{def:saturated}
 is guaranteed by \eqref{eq:realize} above.
 \end{proof}

 \section{The metric case}

 \begin{proposition} Assume \((K,d)\) is a countable bounded metric space 
 and \((X,\rho)=\bigcup_{n<\omega}X_n\) is a countable partition of an
 \(\aleph_1\)-saturated metric space. Then there is an \(n<\omega\) such that
 \(X_n\) contains a scaled copy of \((K,d)\).
 \end{proposition}
 \begin{proof} First fix an enumeration \(\{z_k:k<\omega\}\) of \(K\) and,
 aiming towards a contradiction, assume there is no scaled monochromatic copy
 of \(K\) in \(X\). We shall use the following notation: given an (at most) 
 countable \(Y\subseteq X\) and a function 
 \(f:Y\to\mathbb R^+\) as in \autoref{def:saturated}, let
 \[
  B(Y,f) = \big\{p\in X:(\forall  y\in Y)[d(p,y)=f(y)]\big\}.
 \]
 By our assumption \(B(Y,f)\neq\emptyset\). We shall inductively construct an 
 increasing sequence \(\{Y_n:n<\omega\}\) of finite subsets of \(X\) and 
 functions \(\{f_n:n<\omega\}\) such that 
 \begin{enumerate}
  \item \(f_n\subseteq f_{n+1}\); and
  \item \(f_n:Y_n\to\mathbb R^+\) satisfies \eqref{eq:triangle}; and
  \item \(B(Y_n,f_n)\cap X_i=\emptyset\) for each \(i<n\).
 \end{enumerate}
 Let \(Y_0=f_0=\emptyset\). Assume now that we have constructed \(Y_n,f_n\)
 and choose an arbitrary positive \(c < \min f_n[Y_n]\). 
 (We can choose $c$ because $Y_n$ is finite.) We try to choose $z_i' \in B(Y_n,f_n) \cap X_n$ by induction on $i < \omega$ such that $j < i \Rightarrow \rho(z_j',z_i') = d(z_j,z_i)$. 
 If we succeed then we are done. So without loss of generality there is some $k$ such that 
 $\langle z_i' : i < k\rangle$ is well defined but we cannot choose $z_k'$.
 Let \(K^\prime_n = \{z^\prime_i : i < k\}\) be this copy and let 
 \(Y_{n+1}=Y_n\cup K^\prime_n\). Finally extend \(f_n\) to \(Y_{n+1}\) by defining
 \[
   f_{n+1}(z^\prime_i) = c\cdot d(z_i,z_k).
 \]
 We need to check that \(f_{n+1}\) satisfies \eqref{eq:triangle}. Let $x,y \in \mathrm{dom}(f)$. 
 The condition is easily seen to be satisfied separately on \(Y_n\) (i.e. when $x,y \in Y_n$) by the inductive hypothesis and on \(K^\prime_n\) (i.e. when $x,y \in K_n$) because it is defined from a metric. 
 So without loss of generality let \(y\in Y_n\) and \(x \in K^\prime_n\), so $x = z_i'$ for some 
 $i < k$. Since \(K^\prime_n\subseteq B(Y_n,f_n)\), 
 by definition \(\rho(x,y) = \rho(z_i',y) = f(y)\). But then \eqref{eq:triangle} is clearly satisfied (the triangle 
 is isosceles and the two legs are longer than the base by the choice of $c$).
 
 Finally, we show that the inductive construction has to stop at some point 
 (thus there has to be a scaled copy of $K$ in some $X_n$). Let \(Y=\bigcup_{n<\omega}Y_n\) 
 and \(f=\bigcup_{n<\omega} f_n\). 
 Then \(B(Y,f)\) is nonempty (because $X$ is $\aleph_1$-saturated) and \(B(Y,f)\subseteq B(Y_n,f_n)\) for each \(n<\omega\) 
 (since \(Y_n\subseteq Y\) and $f_n = f\restriction Y_n$). But then \(B(Y,f)\cap X_n=\emptyset\) for each 
 \(n<\omega\)---a contradiction.
 
 \end{proof}

 \section{The Ultrametric Case}

As noted in the introduction, the second author's original motivation for studying these 
questions was the special case
 \[
  \aleph_1\rightarrow_{\mathcal M}(\aleph_0)^1_{\aleph_0}
\]
for the class of bounded metric spaces. Unfortunately, this arrow probably 
does not hold in ZFC. However a modified version of this arrow holds for 
the class of rational ultrametric spaces.

\begin{definition}\label{c0}
1) A metric space $X$ is called \emph{rational} if 
$x,y \in X \Rightarrow \rho(x,y) \in \mathbb{Q}$.

2) Repeating \ref{def:metric}, an \emph{ultrametric} space is a metric space 
that satisfies the strong triangle inequality 
$$\rho(x,z) \leq \max\{\rho(x,y),\rho(y,z)\}.$$ 

3) Given $(X,{\leq})$ a tree and $x,y \in X$, let $\Delta(x,y)$ be 
the $\leq$-maximal $z$ such that $z \leq x \wedge z \leq y$.


4) A tree $T$ is $\theta$-branching iff the set of immediate successors of each element of $T$ is of size $\theta$.
\end{definition}

\begin{theorem}
There is a rational ultrametric space  \((M,d)\) of size \(\aleph_1\)   \label{thm.ultrametric} 
such that for every coloring of \(M\) by countably many colors \(M\) contains 
isometric monochromatic copies of every finite rational ultrametric space.
\end{theorem}

This theorem is both a strengthening and a weakening of the above arrow. On the
one hand we get a universal space for all copies. The price we have to pay is to
restrict the copies to size  \(<\aleph_0\). The proof of the theorem is split
into two parts. We first prove that each finite ultrametric space can be
represented as a special kind of a tree. Then we use a standard  rank-type 
argument to show that whenever the tree \({}^{<\omega}\omega_1\) is colored by
countably many colors it contains monochromatic copies of all finite trees.

Before continuing with the proof of the first part we recall the following
basic observation about ultrametric spaces.

\begin{fact}
Let  \((X,\rho)\) be an ultrametric space. Then every triangle is      \label{fact.isosceles}  
isosceles. Moreover, the base is never longer than the sides. Formally:
\[
(\forall T\in[X]^3)
(\exists \{a,b\}
\subset T
,c\in T\setminus \{a,b\})
\big(\rho(a,b)\leq\rho(a,c)=\rho(b,c)\big)
\]
\end{fact}

\begin{definition*}
A metric space  \((X,\rho)\) is a \emph{rational tree space} if there is 
an ordering \(\leq\) which makes \(X\) a tree and a nonincreasing function 
\(h:X\to\mathbb Q\) such that, for distinct \(x\neq y\in X\),
\[
  \rho(x,y) = \mbox{inf}\ \Big\{h(z):z\leq x\ \&\ z\leq y\Big\}.
\]
We will also call the triple \((X,\leq,h)\) a rational tree space. 
The metric space \((X,\rho)\) is a \emph{rational branch space} if it is a subspace
of a rational tree space $(T,\rho)$ with all nodes of \(X\) being branches (leaf nodes) of 
\((T,\rho)\). It is a \emph{regular rational branch space} if, moreover, each node 
of \(X\) has the same height and the function \(h_T\) is constant on the levels
of \(T\).
\end{definition*}

\begin{proposition*}
Each finite rational ultrametric space is a regular rational branch
space.
\end{proposition*}

\begin{proof}
Let  \((X,\rho)\) be a finite rational ultrametric space. Define a relation
\(\leq_0\) on \(X\) as follows:
\[
  x\leq_0 y \iff (\forall z\neq x)(\rho(x,z)\geq\rho(y,z))
\]

\begin{claim*}
The relation  \(\leq_0\) is transitive.
\end{claim*}

\begin{proof}[Proof of Claim]
Let  \(a\leq_0 b\ \&\ b\leq_0 c\). We need to show that 
\(a\leq_0 c\). We may assume \(a,b,c\) are distinct, otherwise there is nothing
to prove. So consider some \(z\neq a\). We just need to show that 
\(\rho(a,z)\geq\rho(c,z)\). If $z=c$ then $\rho(c,z) = 0 \leq \rho(a,z)$. If \(z=b\), then the inequality follows directly 
from \(b\leq_0 c\). So assume \(z\neq b\). Then 
\(\rho(a,z)\geq\rho(b,z)\geq\rho(c,z)\), and so $\rho(a,z) \geq \rho(c,z)$ as promised. 
The first inequality follows from \(a\leq_0 b\) and the second from \(b\leq_0 c\). 
This finishes the proof of
the claim.\renewcommand{\qedsymbol}{$\blacksquare$}
\end{proof}

\begin{claim*}
For each  \(y\in X\) the set \(\{a:a\leq_0 y\}\) is linearly 
(quasi)-ordered by \(\leq_0\).
\end{claim*}

\begin{proof}[Proof of Claim]
Assume  \(a_0,a_1\leq_0 y\) and, aiming towards a contradiction,
assume that \(a_0\not\leq_0 a_1\) and \(a_1\not\leq_0 a_0\). So there must be 
\(z_0,z_1\) such that \(\epsilon_i = \rho(a_i,z_i) < \rho(z_i,a_{1-i})\) 
for \(i=0,1\). Let $\delta = \rho(a_0,a_1)$. Applying Fact \ref{fact.isosceles} we get 
$\delta = \rho(a_1,a_0) = \rho(a_1,z_0)$ (reading the above inequality for $i=0$) and 
$\delta = \rho(a_0,a_1) = \rho(a_0,z_1)$ (for $i=1$). Now consider the triangle $a_0,z_0,z_1$. We have $\rho(a_0,z_0) < \delta = \rho(a_0,z_1)$ hence by \ref{fact.isosceles} we have $\rho(z_0,z_1) = \delta$.

Since \(a_i\leq_0 y\), we have \(\delta>\rho(a_i,z_i)\geq\rho(y,z_i)\) for $i = 0,1$. But, again
by \ref{fact.isosceles}, the triangle $z_0,z_1,y$ is impossible. This is a contradiction.\renewcommand{\qedsymbol}{$\blacksquare$}
\end{proof}

Consider now the equivalence relation  \(a\simeq b\iff a\geq b\ \&\ b\geq a\) and
refine the \(\leq_0\) order on each equivalence class to an arbitrary linear
order. Call the resulting order \(\leq\). Since \(X\) is finite, it is clear
that \((X,\leq)\) is a tree. For \(s\in X\) put
\[
h(s)=\max\{\rho(s,t):t\geq s\}\ \big(=\max\{\rho(s,t):t\geq_0 s\}\big)
\]
(The second equality follows from the fact that if \(a\simeq b\) and
\(s\neq a,s\neq b\) then \(\rho(s,a)=\rho(s,b)\).) Let \(d\) be the metric 
of the tree space \((X,\leq,h)\).

\begin{claim*}
\(d(x,y)\geq \rho(x,y)\) and, if \(x\leq_0 y\), then \(d(x,y)=\rho(x,y)\) (recalling $d$ is the metric from \ref{thm.ultrametric}).
\end{claim*}

\begin{proof}[Proof of Claim]
Assume first that  \(x\leq_0 y\). Then 
\(d(x,y)=h(x)\geq \rho(x,y)\) by definition. Moreover if \(z\geq_0 x\) then
\(\rho(z,y)\leq\rho(x,y),\rho(x,z)\) (since \(x\leq_0 z\) and \(x\leq_0 y\)) 
and, since \(X\) is ultrametric, it follows that \(\rho(x,y)=\rho(x,z)\). 
In particular, since the choice of \(z\) was arbitrary, \(h(x)=\rho(x,y)\),
proving the second part of the claim. To finish the proof assume now that 
\(x,y\) are incomparable in \(\leq_0\) and let \(s=\Delta(x,y)\).
Then \(h(s)\geq\rho(s,y)\geq\rho(x,y)\) since \(s\leq_0 x\).\renewcommand{\qedsymbol}{$\blacksquare$}
\end{proof}

Unfortunately, the inequality in the above claim can be strict (e.g. if 
we consider the subspace of a tree space which results from deleting a level
the resulting subspace cannot be a tree space). We need to add a
point to the tree for each pair  \(x,y\) with \(\rho(x,y)<d(x,y)\). 
We will use the following claim

\begin{claim*}
Suppose  \((Y,\leq,h)\) is a tree space extending \((X,\leq,h)\) such that 
\(d_Y(x,y)\geq\rho(x,y)\) for each \(x,y\in X\). Suppose that there are
\(a,b\in X\), incompatible in \(\leq\) with \(\rho(a,b)< d_Y(a,b)\). Then there
is a tree space \(Y^\prime\) extending \(Y\) such that
\(d_{Y^\prime}(x,y)\geq\rho(x,y)\) for each \(x,y\in X\) and 
\(\rho(a,b)=d_{Y^\prime}(a,b)\).
\end{claim*}

\begin{proof}[Proof of Claim]
Let  \(Y^\prime= Y\cup\{p\}\) and extend the order so that 
\(\Delta(a,b)\leq p\leq a,b\). Moreover let \(h(p)=\rho(a,b)\). Notice that
if \(x,y\in X\) and either \(x\not\geq a\ \&\ x\not\geq b\) or 
\(y\not\geq a\ \&\ y\not\geq b\) or \(x\leq y\) or \(y\leq x\) then 
\(d_{Y^\prime}(x,y)=d_{Y}(x,y)\) and there is nothing to prove. So, without loss 
of generality, assume \(x\geq a\) and \(y\geq b\). But then 
\(\rho(a,b)\geq\rho(x,b)\) (since \(a\leq x\)) and 
\(\rho(b,x)\geq\rho(x,y)\) (since \(b\leq y\)). Since 
\(\Delta(x,y)=\Delta(a,b)=p\) we have \(d_{Y^\prime}(x,y)=h(p)=\rho(a,b)\) and
this finishes the proof of the claim.\renewcommand{\qedsymbol}{$\blacksquare$}
\end{proof}

Using the above claim to iteratively add points we finally arrive at a 
tree space  \((Y,\leq,h)\) such that \(d_Y\upharpoonright X=\rho\) which,
moreover, has the same distance set as the original \(X\). It is not
hard to further enlarge \(Y\) to make it a regular rational branch space.
\end{proof}

\begin{proposition*}
Assume  \(T\) is an \(\omega_1\)-branching tree\footnote{see \ref{c0}(4)} of height 
\(n<\omega\) and \(\chi:T\to\omega\) is a coloring of the tree by countably
many colors. Then there is an \(\omega_1\)-branching subtree\footnote{meaning the subtree is downward, and all its maximal nodes are maximal nodes of $T$.} of \(T\) whose 
branches (i.e. leaf nodes) have the same color.
\end{proposition*}

\begin{proof}
Given a color  \(c<\omega\) and \(s\in T\) define
\[
  G(s,c,0) \iff \big|\{\alpha:\chi(s^{\smallfrown}\alpha)=c\}\big|=\omega_1
\]
and, inductively, 
\[
  G(s,c,m+1) \iff \big|\{\alpha:G(s^{\smallfrown}\alpha,c,m)\}\big|=\omega_1.
\]
To prove the proposition it is clearly enough to show that there is some
\(c<\omega\) such that \(G(\emptyset,c,\mathrm{ht}(T)-1)\). Suppose otherwise. Then
we can build by induction $\alpha_m$ for $m < \mathrm{ht}(T) - 1$ such 
that for \(m<\mathrm{ht}(T)\) we have 
\[ 
  (\forall c<\omega)\neg G\big(\langle \alpha_i : i < m\rangle,c,\mathrm{ht}(T)-m\big).
\]
For $m=0$ this is our assumption, working towards contradiction. For $m+1$ this is again easy. Hence 
\[
  (\forall c<\omega)\neg G\big(\big\langle \alpha_i:i<\mathrm{ht}(T) - 1\big\rangle,c,0\big)
\]
which is impossible since if we let \(s = \langle \alpha_i : i < \mathrm{ht}(T)-1\rangle\)
then, since \(T\) is \(\omega_1\)-branching, \(s\) must have uncountably
many successors of the same color.
\end{proof}

\begin{proof}[Proof of Theorem  \ref{thm.ultrametric}]
Let  \(M={}^{<\omega}\omega_1\) and define \(h_M:M\to\mathbb Q\) such that
for each \(\sigma\in M\) and each \(q\in [0,h_M(\sigma))\) the set 
\(\{\alpha:h_M(\sigma^{\smallfrown}\alpha)=q\}\) has size \(\aleph_1\). Let \(d_M\) be
the corresponding metric making \(M\) a tree space. Let \(X\) be a finite rational 
metric space, \(h\) a decreasing enumeration of its distance set and let 
\((Y,\leq,h_Y)\) be a tree space witnessing that \(X\) is a regular rational branch 
space. Let \(\chi:T\to\omega\) be an arbitrary coloring of \(M\). Consider the subtree
\(M^\prime=\{s:h\upharpoonright s = d\upharpoonright|s|\}\). Then \(M^\prime\) is 
\(\omega_1\)-branching. By the previous proposition there is a color \(c\) and an 
\(\omega_1\) branching subtree \(M^{\prime\prime}\) of \(M^\prime\) with all 
branches of color \(c\). We can now build an order-isomorphism of \(Y\) into 
\(M^{\prime\prime}\) which, by choice of \(M^\prime\), preserves \(h\). It follows 
that \(M^{\prime\prime}\) contains a monochromatic isometric copy of \(X\).
\end{proof}

\section{acknowledgements}
Several people provided valuable input for the paper. The second author would
like to thank Peter Komj\'{a}th for encouraging discussions and to the members of the
Prague Set Theory seminar, in particular David Chodounsk\'{y} and Jan `Honza' Greb\'{i}k,
who patiently listened to my presentations and helpfully pointed out errors. The
ideas behind the proof of Theorem \ref{thm.main} are entirely due to the first author:
the second author has merely deciphered them and filled in the technical details.

\bibliographystyle{amsalpha}
\bibliography{shlhetal,ramsey-partitions}
\end{document}